\documentclass[12pt,a4paper]{article}

\usepackage{amsmath,amssymb,amsfonts,amsthm}
\usepackage{mathrsfs} 
\usepackage{bm} 
\usepackage{graphicx} 
\usepackage{xcolor}
\usepackage{booktabs}
\usepackage{caption}
\usepackage{hyperref}
\usepackage{geometry} 
\usepackage{algorithm}
\usepackage{algorithmicx}
\usepackage{algpseudocode}

\usepackage{listings}

\theoremstyle{plain}
\newtheorem{theorem}{Theorem}[section] 

\newtheorem{lemma}[theorem]{Lemma}
\newtheorem{corollary}[theorem]{Corollary}

\theoremstyle{definition}
\newtheorem{definition}[theorem]{Definition}
\newtheorem{example}[theorem]{Example}

\theoremstyle{remark}
\newtheorem{remark}[theorem]{Remark}

\title{Necessary and sufficient conditions for relative controllability of discrete linear delay systems}
\author{Javad A. Asadzade\thanks{Department of Mathematics, Eastern Mediterranean University, North Cyprus, Turkey. Email: javad.asadzade@emu.edu.tr} 
	\, and \, Nazim I. Mahmudov\thanks{Department of Mathematics, Eastern Mediterranean University, North Cyprus, Turkey; Research Center of Econophysics, Azerbaijan State University of Economics (UNEC), Baku, Azerbaijan. Email: nazim.mahmudov@emu.edu.tr}}
\date{} 

\begin{document}
	
	\maketitle
	
	\begin{abstract}
		In this paper, we investigate delayed linear difference systems and establish several fundamental results. We first provide a Kalman-type rank condition tailored for delayed linear difference systems. Furthermore, we construct the discrete Gramian matrix and prove its non-singularity, which is essential for analyzing system properties. Additionally, we obtain necessary and sufficient conditions ensuring the relative controllability of the system. Finally, we formulate the control function for effective system control and validate our theoretical findings with numerical examples. These examples illustrate the practical behavior of discrete linear delay systems.
	\end{abstract}
	
	
	\section{Introduction}\label{sec1}

Throughout this paper, we adopt the following notation:  
\begin{itemize}
\item $\Theta$ and $I$ denote the $d \times d$ zero and identity matrices, respectively.
\item For integers $a, b \in \mathbb{Z} \cup \{\pm\infty\}$ with $a \le b$, we write $\mathbb{Z}_a^b := \{a, a+1, \dots, b\}$.
\item The forward difference operator is defined by $\Delta y(r) = y(r+1) - y(r)$.
\end{itemize}

In this article, we study the discrete linear delay system
\begin{equation}\label{eq11}
\begin{cases}
 y(r+1) = Ay(r) + By(r-p) + Cu(r), & r \in \mathbb{Z}_{0}^{\,r_{1}-1}, \\[2mm]
 y(r) = \psi(r), & r \in \mathbb{Z}_{-p}^{\,0}, \\[2mm]
 y(r_{1}) = y^*, & r_{1} \ge r^* \in \mathbb{Z}_{1}^{\,\infty},
\end{cases}
\end{equation}
where $p \ge 1$ is a fixed integer, $A$ and $B$ are constant matrices in $\mathbb{R}^{d \times d}$, and $C \in \mathbb{R}^{d \times k}$.  
The state $y : \mathbb{Z}_{-p}^{\,\infty} \to \mathbb{R}^d$ evolves from a prescribed initial history $\psi : \mathbb{Z}_{-p}^{\,0} \to \mathbb{R}^d$, with $y(r_1) = y^*$ specifying a terminal condition for some finite $r_1 > 0$.  
The control function is $u : \mathbb{Z}_{0}^{\,\infty} \to \mathbb{R}^k$.

The principal aim of this work is to establish the \emph{relative controllability} of \eqref{eq11}.  
A key step in our analysis is to employ an explicit solution representation.  
For systems in which $A$ and $B$ commute, Diblík and Khusainov \cite{Diblik2, Diblik3} developed such a representation for difference equations with constant coefficients and constant delay, using the concept of a \emph{discrete matrix delayed exponential} to extend classical methods from ordinary differential equations to the discrete-delay setting.

This framework was generalized by Mahmudov \cite{Mahmudov2}, who introduced a discrete delayed exponential depending on a sequence of matrices, thereby removing the commutativity requirement.  
Later, in \cite{Mahmudov1}, Mahmudov considered the noncommutative case as follows,
\begin{align}\label{eq2}
\begin{cases}
y(r+1) = Ay(r) + By(r-p) + f(r), & r \geq 1, \\
y(r) = \psi(r), & r \in \mathbb{Z}_{-p}^{0},
\end{cases}
\end{align}
and derived the representation
\begin{align}
y(r) &= Y_{p}^{A,B}(r-p)\psi(0) 
+ \sum_{j=-p}^{-1} Y_{p}^{A,B}(r-1-2p-j)B\psi(j) \nonumber \\
&\quad + \sum_{j=1}^{r} Y_{p}^{A,B}(r-p-j)f(j-1), 
\quad r \in \mathbb{Z}_{-p}^{\infty}.
\end{align}
In the presence of control, the solution to \eqref{eq11} takes the form
\begin{align}\label{16}
y(r) &= Y_{p}^{A,B}(r-p)\psi(0) 
+ \sum_{j=-p}^{-1} Y_{p}^{A,B}(r-1-2p-j)B\psi(j) \nonumber \\
&\quad + \sum_{j=1}^{r} Y_{p}^{A,B}(r-p-j)Cu(j-1).
\end{align}
This representation serves as the analytical foundation for our controllability study.  
In discrete-time systems with delays or fractional dynamics, controllability analysis is often considerably more intricate, typically requiring advanced techniques such as matrix function theory and spectral analysis.

The controllability of delayed discrete systems has been the subject of extensive research.  
Diblík, Khusainov, and Růžičková \cite{Diblik5} established controllability criteria for systems with constant coefficients and pure delay, expressed via the discrete matrix delayed exponential, and explicitly constructed controls steering the state to a prescribed target.  
Diblík, Fečkan, and Pospíšil \cite{Diblik6} devised new control functions for delayed difference equations in boundary value problems, including cases involving invariant subspaces.  
In \cite{Diblik1}, Diblík investigated \emph{trajectory controllability}—the ability to transfer arbitrary solutions to desired trajectories—deriving a Kalman-type criterion and providing illustrative examples.  
Diblík and Mencáková \cite{Diblik4} addressed higher-order delayed systems with a single delay, employing discrete delayed matrix sine and cosine functions to formulate controllability criteria and construct explicit controls, while also showing that certain delay-elimination transformations may yield non-equivalent controllability problems.

Further developments in the theory appear in \cite{Babiarz, Chitour1, Chitour2, Diblik7, Diblik8, Jin, Klamka1, Klamka2, Mahmudov4, Mazanti, Pospisil, Shi, Yang}, reflecting the breadth of current research in the controllability of discrete dynamical systems.

The present work extends and refines the results of \cite{Diblik5} in the setting of \eqref{eq11}.  
Section~\ref{sec2} introduces the necessary preliminary concepts and lemmas.  
In Section~\ref{sec3}, we derive Kalman-type rank conditions for assessing the relative controllability of \eqref{eq11}.  
Section~\ref{sec4} constructs the discrete Gramian matrix, establishes its nonsingularity, and provides necessary and sufficient conditions for controllability, along with an explicit control design.  
Numerical examples illustrating the theoretical results are presented in Section~\ref{sec5}.

\section{Preliminaries}\label{sec2}
In this section, we present key elements, definitions, and lemmas that are essential to establish the relative controllability of \eqref{eq11}. These foundational concepts provide the necessary mathematical framework for the subsequent analysis and proofs.

Now, we define the kernel function \( Y_{p}^{A,B}(r) \), which characterizes the solution of \eqref{eq11}, as follows:
\begin{definition}\cite{Mahmudov1}\label{def1}
Let \(p \geq 1\) and \(A, B\) be constant matrices of dimension \(d \times d\). The delayed perturbation of the discrete matrix exponential is defined as:
\begin{align}
Y_{p}^{A,B}(r) = \sum_{i=0}^{\lfloor \frac{r + p}{p+1} \rfloor} \mathcal{Q}(r + p - pi; i), \, r \in \mathbb{Z}_{0}^{\infty},
\end{align}
where \(\mathcal{Q}(r; i)\) is given by:
\begin{align}
\mathcal{Q}(r; i) =
\begin{cases}
0,& i \in \mathbb{Z}_{-\infty}^{-1}, \\
A^r \delta(r),& i = 0, \\
\sum_{m=i}^{r} A^{r-m}B\mathcal{Q}(m-1; i-1) \delta(r-i),& i \in \mathbb{Z}_{1}^{\infty}.
\end{cases}
\end{align}
Here, \(\delta(r)\) denotes the Heaviside step function defined as:
\[
\delta(r) =
\begin{cases}
0, & r < 0, \\
1, & r \geq 0.
\end{cases}
\]
\end{definition}
This formulation provides a structured representation of solutions to discrete delay systems, emphasizing the role of delayed matrix exponential perturbations.
\begin{remark}
 It should be emphasized that \( \mathcal{Q}(r; i) \) was used in \cite{Mahmudov1} to define delayed perturbation of the discrete matrix exponential  functions. Using the definition \ref{def1} of \( \mathcal{Q}(r; i) \), one may observe the following:
\begin{table}[h!]
\centering
\renewcommand{\arraystretch}{2} 
\setlength{\tabcolsep}{12pt} 
\begin{tabular}{|c|c|c|c|c|c|}
\hline
$\mathcal{Q}(r,i)$ & $i=0$ & $i=1$ & $i=2$ & $\cdots$ & $i=q$ \\ \hline
$\mathcal{Q}(0, i)$ & $I$ & $\Theta$ & $\Theta$ & $\cdots$ & $\Theta$ \\ \hline
$\mathcal{Q}(1, i)$ & $A$ & $B$ & $\Theta$ & $\cdots$ & $\Theta$ \\ \hline
$\mathcal{Q}(2, i)$ & $A^2$ & $AB + BA$ & $B^2$ & $\cdots$ & $\Theta$ \\ \hline
$\vdots$ & $\vdots$ & $\vdots$ & $\vdots$ & $\ddots$ & $\Theta$ \\ \hline
$\mathcal{Q}(q, i)$ & $A^q$ & $\cdots$ & $\cdots$ & $\cdots$ & $B^q$ \\ \hline
\end{tabular}
\caption{Table of $\mathcal{Q}(r,i)$ for different $i$ and $r$}
\label{tab:Q_table}
\end{table}
From the above table, it is easy to see that in the case of commutativity \( AB = BA \), we have:

\[
\mathcal{Q}(r; i) := \binom{r}{i} A^{r-i} B^i \delta(r - i), \, r, i \in \mathbb{Z}_{0}^{\infty}.
\]
\end{remark}
To study relative controllability, we employ an auxiliary result, which is analogous to the fundamental lemma of variational calculus.
\begin{lemma}\label{lem1}\cite{Diblik5}
Let a function \( \Psi : \mathbb{Z}_s^q \to \mathbb{R} \) be given. The equality
\[
\sum_{r=l}^q \Psi(r)\xi(r) = 0,
\]
where \( \xi : \mathbb{Z}_l^q \to R \) is an arbitrary function holds if and only if \( \Psi(r) = 0 \) for all \( r \in \mathbb{Z}_l^q \).
\end{lemma}
And we defined the following lemma, which is important in the proof of main result.
\begin{lemma}\cite{Mahmudov1}\label{lem2}
    $Y_{p}^{A,B}(r)$ is a solution of
    \begin{align*}
       & Y_{p}^{A,B}(r+1)=AY_{p}^{A,B}(r)+BY_{p}^{A,B}(r-p),\\
        &Y_{p}^{A,B}(r)=A^{p+r},\, r\in \mathbb{Z}_{-p}^{0},\, Y_{p}^{A,B}(r)=\Theta,\, r\in \mathbb{Z}_{-\infty}^{-p-1}.
    \end{align*}
\end{lemma}
\begin{lemma}\cite{Mahmudov1}\label{lem3}
    Let $m\in \mathbb{Z}_{0}^{\infty}$. $r\in \mathbb{Z}_{(m-1)(p+1)+1}^{m(p+1)}$ if and only if
    \begin{align*}
        m=\Big\lfloor \frac{r-1}{p+1}\Big\rfloor+1=\Big\lfloor \frac{r+p}{p+1}\Big\rfloor.
    \end{align*}
\end{lemma}
According to \cite{Mahmudov1}, using the Lemma \ref{lem3}, one can easily show that
\begin{align}\label{v67}
Y_{p}^{A,B}(r)=
    \begin{cases}
        \Theta,& r\in \mathbb{Z}_{-\infty}^{-p-1},\\
        \sum_{i=0}^{m}\mathcal{Q}(r+p-pi, i),& r\in \mathbb{Z}^{m(p+1)}_{(m-1)(p+1)+1},\,
        m\in \mathbb{Z}_{0}^{\infty}.
    \end{cases}
\end{align}
The Cayley-Hamilton theorem simplifies the analysis of controllability by expressing matrix powers in terms of lower-order terms. This helps in computing the controllability matrix efficiently and verifying rank conditions. It also aids in deriving the state transition matrix, which is essential for studying state reachability in control systems. The proof of the following theorem is analogous to \cite{Mahmudov4}, and therefore, we omit it.
\begin{theorem}\label{t1}\cite{Mahmudov4}
For any \( n\in \mathbb{Z}_{d}^{\infty} \), the following equality holds:

\begin{align*}
   & \text{rank}\{\mathcal{Q}(r, m)C: \, r\in \mathbb{Z}_{0}^{d-1}\}
    =\text{rank}\{\mathcal{Q}(r, i)C: \, r\in \mathbb{Z}_{0}^{n},\, i\in \mathbb{Z}_{0}^{m}\}.
\end{align*}
\end{theorem}
\begin{remark}
    If \( i=0 \), \( d=k \), and \( C=I \), then we obtain \( \mathcal{Q}(r,0)=A^r \), reducing Theorem \ref{t1} to the Cayley-Hamilton theorem:
    \begin{align*}
        \text{rank}\{A^r\, :\, r\in \mathbb{Z}_{0}^{d-1}\}=  \text{rank}\{A^r\, :\, r\in \mathbb{Z}_{0}^{n}\},\, for\, n\in \mathbb{Z}_{d}^{\infty}.
    \end{align*}
    \end{remark}
\section{Kalman type rank condition}\label{sec3}
In this section of the article, our objective is to derive the \textbf{Kalman type rank condition} for the relative controllability of \eqref{eq11}. To achieve this, we first define the concept of relative controllability as follows:

\begin{definition}\label{def2}\cite{Diblik5}
    System \eqref{eq11} is said to be relatively controllable if, for any initial function \(\psi: \mathbb{Z}^0_{-p} \to \mathbb{R}^d\), any specified terminal state \(y^* \in \mathbb{R}^d\), and any terminal point \(r_1 \geq r^* \in \mathbb{Z}_1^\infty\), there exists a discrete control function \(u^*: \mathbb{Z}^{r_1-1}_0 \to \mathbb{R}^k\) such that when the system is driven by this input \(u = u^*\), i.e.,
\[
y(r+1) = Ay(r)+By(r-p) + C u^*(r),
\]
the system will have a solution \(y^*: \mathbb{Z}^{r_1}_{-p} \to \mathbb{R}^d\), with \(y^*(r_1) = y^*\) and \(y^*(r) = \psi(r)\) for all \(r \in \mathbb{Z}^0_{-p}\). This ensures that the system can be controlled to reach any desired final state starting from any initial condition within a specified time frame.
\end{definition}

Our goal is to apply the \textbf{Kalman type rank condition} to systematically determine whether this relative controllability can be achieved based on the system's dynamics.

Before presenting the criterion for relative controllability, we introduce an auxiliary vector and definitions as follows.
\begin{align}\label{s1}
\eta &=y^* -Y_{p}^{A,B}(r_{1}-p)\psi(0)- \sum_{j=-p}^{-1} Y_{p}^{A,B}(r_{1}-1-2p-j)B\psi(j).
\end{align}
\begin{definition}\cite{Diblik5}
For a given positive number $\delta$, we define the class of bounded discrete functions $\Omega_{\delta}(s, q)$ as:

$$
\Omega_{\delta}(s, q) := \{ \omega: \mathbb{Z}_s^q \to \mathbb{R}^n \mid \|\omega\|_{sq} \leq \delta \}.
$$

Here, for a function $\omega: \mathbb{Z}_s^q \to \mathbb{R}^n$, the norm $\|\omega\|_{sq}$ is given by

$$
\Vert \omega \Vert_{sq} = \max_{j=s, \dots, q} \left(\max_{i=1, \dots, n} \vert \omega_i(j) \vert\right).
$$

\end{definition}

\begin{definition}\cite{Diblik5}
The reachability domain $\mathcal{W}_{\psi} \subseteq \mathbb{R}^d$ is defined as the set of all points $y(r_1)$ attained by solutions of the system \eqref{eq11}, corresponding to a fixed initial condition and to anarbitrary control $u\in \Omega_{\delta}(0, r_1 - 1)$.
\end{definition}
\begin{theorem}\label{uuu}
Problem \eqref{eq11} is relatively controllable if and only if the following two conditions are satisfied:
\begin{align}\label{c1}
    \operatorname{rank}(S) &= d,
\end{align}
and
\begin{align}\label{c2}
    r_1 &\geq r^* = \Big(\frac{d}{k}-1\Big)(p+1)+1,
\end{align}
where \( S = \{\mathcal{Q}(0,i) C,\, \mathcal{Q}(1,i) C,\, \dots ,\,\mathcal{Q}(d-1,i) C : i \in \mathbb{Z}_{0}^{d-1} \)\}, \,  with \( p \geq 1 \) being a positive integer.
\end{theorem}
\begin{proof}
\textbf{Necessity:}  Let the system \eqref{eq11} be relatively controllable. Then, according to Definition \ref{def2}, for any arbitrary initial function \(\psi: \mathbb{Z}_{-p}^{0} \to \mathbb{R}^d\), arbitrary finite terminal state \(y = y^* \in \mathbb{R}^d\), and arbitrary finite terminal point \(r_1\) greater than or equal to a fixed integer \(r^* \in \mathbb{Z}_1^\infty\), there exists a discrete function \(u^*: \mathbb{Z}_0^{r_1-1} \to \mathbb{R}^k\) such that the system \eqref{eq11} has a solution \(y^*: \mathbb{Z}_{-p}^{r_1} \to \mathbb{R}^d\) satisfying condition \(y(r_{1})=y^*\). We show that, in this case, \(\text{rank}(S) = d\) and \(r_1 \geq r^* = (d-1)(p+1)+1\) are true. We use the formula \eqref{16} is employed to represent the solution \(y^*\) of the problem \eqref{eq11}. At the moment \(r = r_1\), we have
\begin{align}\label{o1}
y^*&=y^*(r_{1}) = Y_{p}^{A,B}(r_{1}-p)\psi(0) + \sum_{j=-p}^{-1} Y_{p}^{A,B}(r_{1}-1-2p-j)B\psi(j)\nonumber\\
&+ \sum_{j=1}^{r_{1}} Y_{p}^{A,B}(r_{1}-p-j)Cu^*(j-1).
\end{align}
By using \eqref{s1}, we rewrite \eqref{o1} as follows
\begin{align}\label{o2}
 \sum_{j=1}^{r_{1}} Y_{p}^{A,B}(r_{1}-p-j)Cu^*(j-1)=\eta.
\end{align}
Since $r_{1}\geq 1$, and using the definition \ref{def1}, we get
\begin{align}\label{o3}
 \sum_{j=1}^{r_{1}}\Bigg[\sum_{i=0}^{\lfloor \frac{r_{1}-j }{p+1} \rfloor} \mathcal{Q}(r_{1} -j - pi; i) \Bigg] Cu^*(j-1)=\eta.
\end{align}
The left-hand side represents a linear combination of \( m + 1 \) vectors:
\[
\mathcal{Q}(0, i)C,\, \mathcal{Q}(1, i)C,\, \dots,\, \mathcal{Q}(m-1, i)C,\, \mathcal{Q}(m, i)C,
\]
where \( m\in \mathbb{Z}_{0}^{\infty} \), $i\in \mathbb{Z}_{0}^{m}$. We now determine the possible values of \( m \). The system of \( d \) equations given in \eqref{o3} can be reformulated in a simplified form as:
\begin{align}\label{Q1}
 \mathcal{Q}(0, i)C \lambda^{0}+  \mathcal{Q}(1, i)C\lambda^{1} + \dots +  \mathcal{Q}(m, i)C\lambda^{m} = \eta,
\end{align}
where the unknown vectors \( \lambda^{0}, \dots, \lambda^{m} \) are to be determined as follows:
\begin{align*}
   \lambda^{0}= \begin{bmatrix} \lambda^{0}_{1} \\\vdots\\
   \lambda^{0}_{k} \end{bmatrix}, \lambda^{1}= \begin{bmatrix} \lambda^{1}_{1} \\\vdots\\
   \lambda^{1}_{k} \end{bmatrix}.\dots,  \lambda^{m}= \begin{bmatrix} \lambda^{m}_{1} \\\vdots\\
   \lambda^{m}_{k} \end{bmatrix}.
\end{align*}
So that, we have
\begin{align*}
  \begin{bmatrix} \eta_{1} \\\vdots\\
   \eta_{d} \end{bmatrix}  & =\mathcal{Q}(0, i)\begin{bmatrix} \sum_{j=1}^{k} c_{1j}\lambda^{0}_{j} \\\vdots\\
  \sum_{j=1}^{k} c_{dj}\lambda^{0}_{j} \end{bmatrix}+\dots+\mathcal{Q}(m, i)\begin{bmatrix} \sum_{j=1}^{k} c_{1j}\lambda^{m}_{j} \\\vdots\\
  \sum_{j=1}^{k} c_{dj}\lambda^{m}_{j} \end{bmatrix}.
\end{align*}
A simplified version of equation \eqref{Q1} can be written as
\begin{align*}\begin{cases}
    \eta_{1}=\sum_{j=1}^{k}\Big(\sum_{l=1}^{d}\mathcal{Q}_{1l}(0,i)c_{lj}\Big)\lambda^{0}_{j}
    +\dots+\sum_{j=1}^{k}\Big(\sum_{l=1}^{d}\mathcal{Q}_{1l}(m,i)c_{lj}\Big)\lambda^{m}_{j},\\
     \eta_{2}=\sum_{j=1}^{k}\Big(\sum_{l=1}^{d}\mathcal{Q}_{2l}(0,i)c_{lj}\Big)\lambda^{0}_{j}
    +\dots+\sum_{j=1}^{k}\Big(\sum_{l=1}^{d}\mathcal{Q}_{2l}(m,i)c_{lj}\Big)\lambda^{m}_{j},\\
    \dots\dots\dots\dots\dots\dots\dots\dots\dots\dots \dots\dots\dots\dots\dots\dots\dots\dots\dots\dots\\
    \eta_{d}=\sum_{j=1}^{k}\Big(\sum_{l=1}^{d}\mathcal{Q}_{dl}(0,i)c_{lj}\Big)\lambda^{0}_{j}
    +\dots+\sum_{j=1}^{k}\Big(\sum_{l=1}^{d}\mathcal{Q}_{dl}(m,i)c_{lj}\Big)\lambda^{m}_{j}.
    \end{cases}
\end{align*}
It is important to note that \( \lambda^{0}, \dots, \lambda^{m} \) vectors depend on the values \( u^{*}(0), \dots, u^{*}(r_{1}-1) \), although this dependency is not explicitly considered here.  If \( (m+1)k < d \), the system of \( d \) equations given in \eqref{o3} with unknown vectors \( \lambda^{0}, \dots, \lambda^{m} \) becomes overdetermined. In such cases, a solution does not always exist if \( \eta \) is arbitrarily chosen. This leads to the necessary condition:
\begin{align}\label{t33}
(m+1)k \geq d
\end{align}
for \eqref{Q1} to admit a solution.

Consider \eqref{t33}, \eqref{o3}, and \eqref{Q1}. Looking for a maximal possible value of $m$, we see that the maximum value of $i$ in \eqref{o3}, which is the maximal value of $m$, is defined by the number
\begin{align}
    \Big\lfloor \frac{r_{1}-j}{p+1}\Big\rfloor.
\end{align}
If $j=1$. Therefore, from inequality \eqref{t33}, that is, from
\begin{align}
    m= \Big\lfloor \frac{r_{1}-1}{p+1}\Big\rfloor\geq \frac{d}{k}-1.
\end{align}
we get $r_{1}\geq(\frac{d}{k}-1)(p+1)+1=r^*$, and  \eqref{c2} holds. According to the Theorem \ref{t1} , any function of the form \(\mathcal{Q}(j, i)\) for \(j \geq d\) can be represented as a linear combination of the functions:
\[
\mathcal{Q}(0, i)C,\, \mathcal{Q}(1, i)C,\, \dots,\,  \mathcal{Q}(d-1, i)C,\, \text{for},\, i\in \mathbb{Z}_{0}^{d-1}.
\]
Thus, the aforementioned system \eqref{Q1} of \( d \) linear equations has a solution for any arbitrary right-hand side \( \eta \) if and only if \( \operatorname{rank}(S) = d \). As a result, condition \eqref{c1} is satisfied.

\textbf{Sufficiency:} Let \eqref{c1} and \eqref{c2} hold. We demonstrate that the system \eqref{eq11} is relatively controllable. To begin, we establish that the dimension of the reachability domain is equal to \( d \). Assume the contrary, i.e., \( \dim \mathcal{W}_{\psi} < d \). Then there exists a nontrivial constant vector \( z = (z_1, z_2, \dots, z_d)^T \in \mathbb{R}^d \) such that, for any control \( u \in \Omega_{\delta}(0, r_1-1) \) and the corresponding solution \( y = y(r; u, \psi) \) of the problem \eqref{eq11}, we have
\begin{align}\label{14}
z^T y(r_1) = 0.
\end{align}
We choose the zero initial function \( \psi(r) = 0, \, r \in \mathbb{Z}_0^{-p}. \) Then, with the aid of formula \eqref{16}, we obtain:
\begin{align}\label{23}
y(r_{1}) = \sum_{j=1}^{r_{1}} Y_{p}^{A,B}(r_{1}-p-j)Cu(j-1),
\end{align}
 and \eqref{14} becomes
 \begin{align}\label{i9}
z^T \sum_{j=1}^{r_{1}} Y_{p}^{A,B}(r_{1}-p-j)Cu(j-1)=0.
\end{align}
Since \eqref{i9} is valid for any arbitrary control \( u \in \Omega_{\delta}(0, r_1 - 1) \), applying Lemma \ref{lem1} gives the result:
\begin{align}\label{w1}
z^T  Y_{p}^{A,B}(r_{1}-p-j)C = 0, \, j \in \mathbb{Z}_0^{r_1}.
\end{align}
 Now we apply the lemma \ref{lem2} to \eqref{w1}. We obtain
 \begin{align}\label{w2}
&z^T  Y_{p}^{A,B}(r_{1}-p-j)C = z^{T}(A Y_{p}^{A,B}(r_{1}-p-j-1)+B Y_{p}^{A,B}(r_{1}-2p-j-1))C=0,
\end{align}
or
\begin{align}\label{w3}
  z^{T} \sum_{l=0}^{1} \mathcal{Q}(1,l) Y_{p}^{A,B}(r_{1}-(l+1)p-j-1)C=0,\, j \in \mathbb{Z}_0^{r_1}.
\end{align}
We repeatedly apply the same procedure to \eqref{w3} a total of \( d-2 \) times. Then, according to Lemma \ref{lem2}, for $ j \in \mathbb{Z}_0^{r_1},$ we have:
\begin{align}\label{w4}
& z^{T} \sum_{l=0}^{2} \mathcal{Q}(2,l) Y_{p}^{A,B}(r_{1}-(l+1)p-j-2)C=0,\\
& z^{T} \sum_{l=0}^{3} \mathcal{Q}(3,l) Y_{p}^{A,B}(r_{1}-(l+1)p-j-3)C=0,\label{w5}
 \\
   & \dots\dots\dots \dots\dots\dots\dots\dots\dots \dots\dots\dots\dots\dots\dots \dots\dots\dots\nonumber
   \\
& z^{T} \sum_{l=0}^{d-1} \mathcal{Q}(d-1,l) Y_{p}^{A,B}(r_{1}-(l+1)p-j-d+1)C=0,\label{w6}\end{align}
We substitute \( j = r_1 \) into \eqref{w1}, \( j = r_1 - lp -1\) into \eqref{w3}, and proceed step by step with \( j = r_1 - lp-2, j = r_1 - lp-3,\dots, j = r_1 - lp-d+1 \) in \eqref{w4}, \eqref{w5}, and \eqref{w6}.  Since inequality \( r_1 \geq (d-1)(p+1)+1 \) is valid due to \eqref{c2}, we have \( r_1 - (d-1)p \geq d \geq 1 \), ensuring that all the chosen indices \( j \) are valid.
By Lemma \ref{lem2}, we have \( Y_{p}^{A,B}(-p) = I \). Consequently, the system \eqref{w1}–\eqref{w6} (with respect to \( b \)) simplifies to:
\begin{align}\label{w7}
    z^{T} \mathcal{Q}(0,l)C=0, \dots,\, z^{T} \mathcal{Q}(d-1,l)C=0,\, \text{for}\, l\in \mathbb{Z}_{0}^{d-1}.
\end{align}
Using the definition of the matrix \( S \), we rewrite \eqref{w7} as
\begin{align}\label{w8}
z^T S = 0.
\end{align}
Thus the matrix \( S \) has no full rank. This contradicts condition \eqref{c1}. Therefore, the inequality \( \dim \mathcal{W}_\psi < d \) does not hold, and the dimension of the reachability domain is equal to \( d \).

Since the domain of reachability \( \mathcal{W}_\psi \) includes the point \( -y(r_1) \) corresponding to a control \( -u \in \Omega_{\delta}(0, r_1 - 1) \) along with the point \( y(r_1) \) corresponding to a control \( u \in \Omega_{\delta}(0, r_1 - 1) \) (based on the choice of a zero initial function and the representation in \eqref{23}, we conclude that \( \mathcal{W}_\psi \) is symmetric. Furthermore, the domain of reachability contains the segment connecting the point \( y(r_1) \) to \( -y(r_1) \).

As a result, due to the linearity of the problem, it also contains a ball of radius \( \varepsilon \).
 \begin{align}\label{w9}
      U_\varepsilon = \{ y \in \mathbb{R}^d \mid \|y\| < \varepsilon \}
       \end{align}
 with positive \( \varepsilon \). Clearly, if \( \delta \to \infty \) in \( \Omega_{\delta}(0, r_1 - 1) \), then, due to the finiteness of the interval \( \mathbb{Z}_0^{r_1 - 1} \), it follows that \( \varepsilon \to \infty \) in the definition of \( U_\varepsilon \). As a result, the domain \( \mathcal{W}_\psi \) spans the entire space \( \mathbb{R}^d \). This implies that for every point \( y^* \in \mathbb{R}^d \), there exists a control \( u = u^* \) that solves the problem \eqref{eq11}. This conclusion remains valid even if the initial function \( \psi \) is nonzero.
Indeed, by introducing the transformation \( y(r) = y_\psi(r) + \xi(r) \), where \( y_\psi(r) \) is a solution of the homogeneous problem
\begin{align*}
    \begin{cases}
        y_\psi(r+1) =A y_{\psi}(r)+ B y_\psi(r - p), & r \in \mathbb{Z}_0^{r_1 - 1},\\
y_\psi(r) = \psi(r), & r \in \mathbb{Z}_{-p}^0,
    \end{cases}
\end{align*}
the problem is reduced to one with respect to \( \xi \), which has a zero initial function.
\end{proof}

We obtain the following result as a special case.

\begin{corollary}\cite{Diblik5}
    When \( A = I \) and \(C=b\) in \eqref{eq11}, the problem \eqref{eq11} is relatively controllable if and only if the following two conditions are satisfied:
\begin{align}\label{c11}
    \operatorname{rank}(S) &= d,
\end{align}
and
\begin{align}\label{c22}
    r_1 &\geq r^* = (d-1)(p+1)+1,
\end{align}
where \( S = \{ Ib,\, B b,\, \dots ,B^{d-1} b \} \), with \( p \geq 1 \) being a positive integer.
\end{corollary}
\section{A control function for the delay discrete system}\label{sec4}
This section constructs control function $u^*$ for the discrete-time delay system governed by equation \eqref{eq11}. The development of this control law relies on establishing relative
controllability and solving a system involving a controllability Gramian. The core results are summarized below with their respective proofs.

\begin{lemma}\label{re}
The system \(\eqref{eq11}\) is relatively controllable if and only if

\begin{align}\label{qre}
\eta^{\top}Y_{p}^{A,B}(r)C=0 \quad \text{for all}\quad  r\in \mathbb{Z}_{-\infty}^{r^*-1}\implies \eta=0.
\end{align}
\end{lemma}
\begin{proof}
\textbf{Sufficiency:} Assume the condition \eqref{qre} holds and the
 only solution is $\eta=0$. This implies the matrix
   \[
   S = \{\mathcal{Q}(0,i) C,\, \mathcal{Q}(1,i) C,\, \dots, \mathcal{Q}(d-1,i) C : i \in \mathbb{Z}_{0}^{d-1} \}
   \]
 has full rank: $  \operatorname{rank}(S) = d$, ensuring the system is relatively controllable.
 
\textbf{Necessity:} Suppose the system is relatively controllable and $\eta \neq 0$ satisfies $\eta^\top Y^{A,B}_{p}(r) C = 0$. 
Then $\eta \in \ker(S^\top)$, which contradicts the full-rank assumption of $S$. 
Thus, $\eta = 0$.
\end{proof}
\begin{lemma}\label{qqq}
 For \( r_1 \geq r^* \), the discrete controllability
 Gramian
\[
\Gamma = \sum_{j=1}^{r_1} Y_{p}^{A,B}(r_1 - p - j) C C^{\top} \big(Y_{p}^{A,B}(r_1 - p - j) \big)^{\top} 
\]
 is nonsingular (i.e., positive definite).
\end{lemma}
\begin{proof} For any nonzero vector $\eta\in \mathbb{R}^d$, consider
\begin{align*}
    v_j=\left(Y_{p}^{A,B}(r_1 - p - j) C\right)^{\top} \eta.
\end{align*}
Since the columns of $Y^{A,B}_{p}(r)C$ are linearly independent (by Lemma \ref{re}), $v_j \neq 0$ for at least one $j$. Thus,
\[
\eta^\top \Gamma \eta = \sum_{j=1}^{r_1} \langle v_j^\top ,v_j \rangle = \sum_{j=1}^{r_1} \|v_j\|^2 > 0.
\]
Hence, $\Gamma$ is positive definite $\Rightarrow \det(\Gamma) \neq 0$.
\end{proof}
\begin{theorem}\label{rr4} If the system is relatively controllable, then
 the  control function solving equation \eqref{eq11} is given by:
\begin{align*}
    u^*(r) = C^{\top} \big(Y^{A,B}_{p}(r_{1} - p - r - 1)\big)^{\top} \Gamma^{-1} \eta,\quad r \in  \mathbb{Z}_{0}^{r_{1}-1}
\end{align*}
\end{theorem}
\begin{proof}
 The control must satisfy:
\begin{align*}
 \sum_{j=1}^{r_{1}} Y_{p}^{A,B}(r_{1}-p-j)Cu(j-1)=\eta .
\end{align*}
 Assume a parametric form:
\begin{align*}
    u(j-1) = \left(Y_{p}^{A,B}(r_{1}-p-j)C\right)^{\top} \mathcal{E}.
\end{align*}
 Substituting into the above equation yields:
\begin{align*}
    \Gamma \mathcal{E}=\eta\implies \mathcal{E} = \Gamma^{-1} \eta.
\end{align*}
Substituting back, the control law becomes:
\begin{align*}
u^*(r) = C^{\top}\left(Y^{A,B}_{p}(r_{1} - p - r - 1)\right)^{\top} \Gamma^{-1} \eta.
\end{align*}
\end{proof}
\section{Numerical examples}\label{sec5}
In this section, we present numerical examples to reinforce our theoretical findings on the relative controllability of discrete-time delay systems described by \eqref{eq11}. 
\begin{example}
    Consider the following discrete-time delay system:
\begin{equation} \label{eq_example}
\begin{cases}
 y(r+1) = \begin{bmatrix} 1 & 2 \\ 0 & 1 \end{bmatrix} y(r) + \begin{bmatrix} 0 & 1 \\ 1 & 0 \end{bmatrix} y(r-1) + \begin{bmatrix} 1 \\ 0 \end{bmatrix} u(r), \, r \in \mathbb{Z}_{0}^{2}, \\
y(-1) = \begin{bmatrix} 1 \\ 0 \end{bmatrix}, \quad y(0) = \begin{bmatrix} 2 \\ 1 \end{bmatrix}, \\
y(3) = \begin{bmatrix} 21 \\ 14 \end{bmatrix}.
\end{cases}
\end{equation}
According to Theorem \ref{uuu}, we define $S$ as follows:
\[
S = \{\mathcal{Q}(0,0) C\, \mathcal{Q}(1,0) C\, \mathcal{Q}(1,0) C\} = \{C \, AC \, BC\}.
\]
Substituting the values of \( A,\, B \), and \(C\)  we obtain:
\[
S = \Bigg\{  \begin{bmatrix} 1 \\ 0 \end{bmatrix}\,  \begin{bmatrix} 1 \\ 0 \end{bmatrix}\, \begin{bmatrix} 0 \\ 1 \end{bmatrix} \Bigg\}\implies \, \text{rank}(S)=2.
\]
On the other hand, we compute:
\[
d = 2,\, k=1\quad \text{and}\quad r_1 = 3.
\]
Checking the second condition of Theorem \ref{uuu}:
\begin{align}
   r_1 &\geq r^* = \Big(\frac{d}{k}-1\Big)(p+1)+1,
\end{align}
we find that this condition is satisfied.

Therefore, based on Theorem \ref{uuu}, the system \eqref{eq_example} is relatively controllable.
\end{example}
\begin{example}
In this second example, we examine the same discrete-time delay system. Our objective is to illustrate the application of the lemma \ref{qqq} by calculating the discrete Gramian matrix under the conditions specified in Example 1.

The discrete Gramian matrix is defined as:
\begin{align*}
\Gamma &= \sum_{j=1}^{3} \Big[ Y_{1}^{A,B}(2 - j) C C^{\top} \big(Y_{1}^{A,B}(2 - j) \big)^{\top} \Big].
\end{align*}
Using the definition of \( Y_{p}^{A,B}(r) \) from equation \eqref{v67}, we determine the values of \( Y_{1}^{A,B}(1) \), \( Y_{1}^{A,B}(0) \), and \( Y_{1}^{A,B}(-1) \) as follows:
\begin{align*}
\begin{cases}
 Y_{1}^{A,B}(1) = \mathcal{Q}(2,0) + \mathcal{Q}(1,1) = A^2 + B,\\
 Y_{1}^{A,B}(0) = \mathcal{Q}(1,0) = A,\\
 Y_{1}^{A,B}(-1) = \mathcal{Q}(0,0) = I.
\end{cases}
\end{align*}
Substituting the given system matrices, we obtain:
\begin{align*}
\begin{cases}
     Y_{1}^{A,B}(1) C  C^{\top} \big(Y_{1}^{A,B}(1) \big)^{\top} = \begin{bmatrix} 1 & 1 \\ 1 & 1 \end{bmatrix}, \\
     Y_{1}^{A,B}(0) C  C^{\top} \big(Y_{1}^{A,B}(0) \big)^{\top}= \begin{bmatrix} 1 & 0 \\ 0 & 0 \end{bmatrix},\\
       Y_{1}^{A,B}(-1) C  C^{\top} \big(Y_{1}^{A,B}(-1) \big)^{\top} = \begin{bmatrix} 1 & 0 \\ 0 & 0 \end{bmatrix}.
     \end{cases}
\end{align*}
Now, calculating the Gramian matrix:
\begin{align*}
    \Gamma = \begin{bmatrix} 3 & 1 \\ 1 & 1 \end{bmatrix}\implies  \det(\Gamma) = 2 \neq 0.
\end{align*}
Since the determinant of \( \Gamma \) is non-zero, we conclude that the discrete Gramian matrix is non-singular.
\end{example}
\begin{example}
In this example, we illustrate  \eqref{eq_example} using Lemma \ref{qqq}, Theorem \ref{rr4}, and the solution \eqref{16}. By Lemma \ref{lem2},  
\[
Y_{1}^{A,B}(2)=AY_{1}^{A,B}(1)+BY_{1}^{A,B}(0)=A^3 +AB+BA=\begin{bmatrix} 3&8 \\ 2&3 \end{bmatrix}.
\]  
From \eqref{s1}, we obtain  
\[
\eta =y^* -Y_{1}^{A,B}(2)\psi(0)-  Y_{1}^{A,B}(1)B\psi(-1)=\begin{bmatrix} 2 \\ 6 \end{bmatrix}.
\]  
The control function is  
\[
u^*(r) = C^{\top} \big(Y^{A,B}_{1}(1 - r)\big)^{\top} \Gamma^{-1} \eta, \quad r \in \mathbb{Z}_{0}^{2},
\]  
which yields $u^*(0)=6$ and $u^*(1)=-2$. Using \eqref{16}, we get  
\[
y(1)=\begin{bmatrix} 10 \\ 2 \end{bmatrix}, \quad y(2)=\begin{bmatrix} 13 \\ 4 \end{bmatrix}.
\]  

On the other hand, there exist infinitely many control functions solving \eqref{eq_example}. Table~\ref{tab:my_table} lists six examples, and Fig.~\ref{fig:trajectories} shows the corresponding trajectories. Each solution $y(r)=(y_1(r),y_2(r))^{\top},\, r\in \mathbb{Z}_{-1}^{3}$, is visualized as points fitted with smooth lines. The blue points mark the initial values, the red point indicates the terminal state, and the solution corresponding to $u^*$ is shown with a blue curve.  

\begin{table}[h!]
\centering
\renewcommand{\arraystretch}{1.8} 
\setlength{\tabcolsep}{10pt} 
\begin{tabular}{|c|c|c|c|c|c|}
\hline
 & $u(0)$ & $u(1)$ & $y(1)$ & $y(2)$ & $y^*=y(3)$ \\ \hline
$\mathbf{r}$ & $\mathbf{6}$ & $\mathbf{-2}$ & $\mathbf{(10,2)^{\top}}$ & $\mathbf{(13,4)^{\top}}$ & $\mathbf{(21,14)^{\top}}$ \\ \hline
p & 6 & 3 & $(10,2)^{\top}$ & $(18,4)^{\top}$ & $(21,14)^{\top}$ \\ \hline
q & 6 & 5 & $(10,2)^{\top}$ & $(20,4)^{\top}$ & $(21,14)^{\top}$ \\ \hline
v & 6 & $-12$ & $(10,2)^{\top}$ & $(3,4)^{\top}$ & $(21,14)^{\top}$ \\ \hline
w & 6 & $\frac{1}{2}$ & $(10,2)^{\top}$ & $(15.5,4)^{\top}$ & $(21,14)^{\top}$ \\ \hline
l & 6 & $-\frac{3}{2}$ & $(10,2)^{\top}$ & $(13.5,4)^{\top}$ & $(21,14)^{\top}$ \\ \hline
\end{tabular}
\caption{Control functions and corresponding solutions}
\label{tab:my_table}
\end{table}

\begin{figure}[h!]
    \centering
    \includegraphics[width=0.8\textwidth]{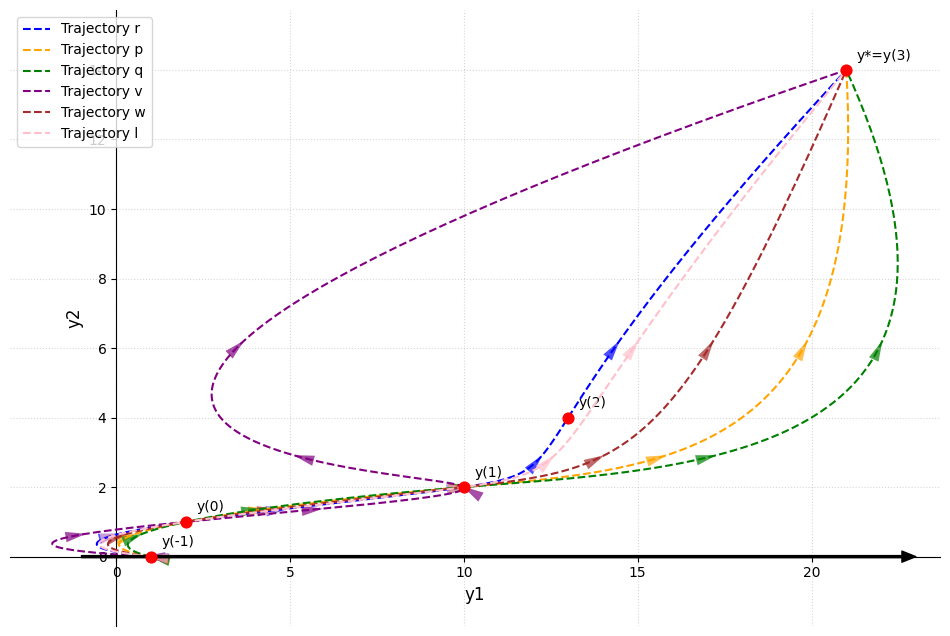}
    \caption{Trajectories of \eqref{eq_example} for different control functions.}
    \label{fig:trajectories}
\end{figure}

\end{example}

\section{Conclusion}

In this work, we have made the following contributions:  

\begin{itemize}
\item Investigated discrete-time linear difference systems with delay and established a Kalman-type rank condition for relative controllability.
\item Demonstrated the non-singularity of the discrete Gramian matrix, a key property in the analysis of system behavior.
\item Derived necessary and sufficient conditions ensuring the relative controllability of the system.
\item Constructed an explicit control function enabling effective regulation of the system state.
\item Provided numerical examples illustrating the applicability and validity of the theoretical results.
\end{itemize}

\end{document}